\documentclass[reqno]{amsart}

\usepackage{booktabs}
\usepackage{xcolor}
\usepackage{mathtools}
\mathtoolsset{showonlyrefs, showmanualtags}
\usepackage{aligned-overset}
\usepackage{enumitem}
\usepackage{mleftright}

\numberwithin{equation}{section}

\newtheorem{theorem}{Theorem}[section]
\newtheorem{lemma}[theorem]{Lemma}

\newtheorem{corollary}[theorem]{Corollary}
\theoremstyle{remark}
\newtheorem{remark}[theorem]{Remark}

\usepackage[colorlinks=true,linkcolor=blue,citecolor=blue,urlcolor=blue]{hyperref}

\begin{document}

\title[Sharp bounds for modified Bessel functions]{Sharp two-sided bounds for the product of modified Bessel functions and for its logarithmic derivative}

\author{Javier Segura}
\address[J.~Segura]{Departamento de Matem\'aticas, Estad\'{\i}stica y Computaci\'on, Universidad de Cantabria, 39005 Santander, Spain.}
\email{javier.segura@unican.es}
\thanks{J. Segura acknowledges support from Ministerio de Ciencia, Innovaci\'on y Universidades, project PID2024-159583NB-I00 (MCIN/AEI/10.13039/501100011033/FEDER, UE)}

\author{Soichiro Suzuki}
\address[S.~Suzuki]{Department of Mathematics, Chuo University, 1-13-27 Kasuga, Bunkyo-ku, Tokyo 112-8551, Japan.}
\email{soichiro.suzuki.m18020a@gmail.com}

\subjclass[2020]{Primary 33C10; Secondary 26D07, 34A40}

\date{}

\dedicatory{}

\begin{abstract}
	We prove that, for all $x > 0$, the product of modified Bessel functions satisfies
	\[
	\frac{1}{2\sqrt{x^2+\nu^2+1/5}}<I_\nu(x)K_\nu(x)<\frac{1}{2\sqrt{x^2+\nu^2-1/4}} ,
	\]
	where the lower bound holds for $\nu>-1$ and the upper bound for $\nu\ge 1/2$, and that the sum of the logarithmic derivatives satisfies
	\[
	-\frac{x}{x^2+\nu^2-1}<\frac{I_\nu'(x)}{I_\nu(x)}+\frac{K_\nu'(x)}{K_\nu(x)}<-\frac{x}{x^2+\nu^2+7/20},
	\]
	where the upper bound holds for $\nu>-1$ and the lower bound for $\nu\ge 1$. All four constants are best possible. The four proofs share the same elementary structure: in each case the relevant auxiliary function satisfies a second-order differential inequality, and the inequality follows from the strong maximum principle. Among other applications, we obtain sharp monotonicity properties of the product and improved bounds for the ratios $I_{\nu-1}(x)/I_\nu(x)$ and $K_{\nu+1}(x)/K_\nu(x)$.
\end{abstract}

\maketitle

\section{Introduction}

The product $P_\nu(x)=I_\nu(x)K_\nu(x)$ of modified Bessel functions and related quantities appear in a considerable number of applications (see, for instance, \cite{Ellingsen:2010,Fucci:2024,Gigante:2013,Ikoma:2025,Truc:2012} and the references cited in \cite{Segura:2021,SeguraSuzuki:2026}). Bounds for $P_\nu(x)$ have been obtained by several authors \cite{Baricz:2016,Gaunt:2016,Segura:2021,SeguraSuzuki:2026}, mostly as consequences of bounds for the ratios $I_{\nu-1}(x)/I_\nu(x)$ and $K_{\nu+1}(x)/K_\nu(x)$.

In \cite{SeguraSuzuki:2026} it was proved, by means of a differential argument, that
\begin{equation}\label{eq:SSlower}
	I_\nu(x)K_\nu(x)>\frac{1}{2\sqrt{x^2+\nu^2+1/5}},\qquad x>0,\ \nu>-1,
\end{equation}
and that the constant $1/5$ cannot be lowered. This inequality had been conjectured in \cite{Segura:2021}, and it cannot be obtained from the known bounds for the ratios because it is extremely sharp in the direction $x=\sqrt{2/3}\,\nu$, $\nu\to+\infty$ (see \cite{SeguraSuzuki:2026}).

The purpose of this paper is to show that the idea behind the proof of \eqref{eq:SSlower} is more general. First, the same differential equation gives the companion upper bound
\begin{equation}\label{eq:upperprod}
	I_\nu(x)K_\nu(x)<\frac{1}{2\sqrt{x^2+\nu^2-1/4}},\qquad x>0,\ \nu\ge\tfrac12,
\end{equation}
with the best possible constant $-1/4$ and range $\nu \geq 1/2$. Therefore, writing
\begin{equation}\label{eq:defh}
	q(x)=\frac{1}{4P_\nu(x)^2},\qquad h(x)=q(x)-x^2-\nu^2,
\end{equation}
we have the two-sided estimate $-1/4<h(x)<1/5$ for $\nu\ge 1/2$.

Second, because of the Wronskian relation \cite[10.28.2]{NIST:DLMF} the product is related to the \emph{difference} of the logarithmic derivatives,
\begin{equation}\label{eq:diff}
	\ell_I(x)-\ell_K(x)=\frac{1}{x\,I_\nu(x)K_\nu(x)},\qquad \ell_I=\frac{I_\nu'}{I_\nu},\quad \ell_K=\frac{K_\nu'}{K_\nu},
\end{equation}
while the \emph{sum} of the logarithmic derivatives is the logarithmic derivative of the product,
\begin{equation}\label{eq:sum}
	\ell_I(x)+\ell_K(x)=\frac{d}{dx}\log\mleft(P_\nu(x)\mright) = \frac{P_\nu'(x)}{P_\nu(x)}.
\end{equation}
We prove that, with
\begin{equation}\label{eq:defH}
	Q(x)=- \frac{x P_\nu(x)}{P_\nu'(x)}, \qquad H(x) = Q(x) - x^2 - \nu^2,
\end{equation}
one has $H(x) < 7/20$ for $\nu>-1$ and $H(x)>-1$ for $\nu\ge1$, again with best possible constants. 

The main results are summarized in the following two theorems.
\begin{theorem}[Product]\label{thm:product}
	Let $x>0$. Then we have
	\begin{alignat}{2}
		I_\nu(x)K_\nu(x) &> \frac{1}{2\sqrt{x^2+\nu^2+1/5}} ,\qquad& &\nu>-1
		\label{eq:product-lowerbound}
		\shortintertext{and}
		I_\nu(x)K_\nu(x)&<\frac{1}{2\sqrt{x^2+\nu^2-1/4}},\qquad& &\nu\ge\tfrac12 .
		\label{eq:product-upperbound}
	\end{alignat}
	Here, the constant $1/5$ in \eqref{eq:product-lowerbound} is the smallest possible such that the inequality holds for all $x > 0$ and all $\nu > -1$,
	and the constant $-1/4$ in \eqref{eq:product-upperbound} is the largest possible such that the inequality holds for all $x > 0$ with fixed $\nu \geq 1/2$.
\end{theorem}

\begin{theorem}[Sum of logarithmic derivatives]\label{thm:sum}
	Let $x>0$. Then we have
	\begin{alignat}{2}
		\frac{I_\nu'(x)}{I_\nu(x)}+\frac{K_\nu'(x)}{K_\nu(x)}&<-\frac{x}{x^2+\nu^2+7/20},\qquad& &\nu>-1
		\label{eq:sum-upperbound}
		\shortintertext{and}
		\frac{I_\nu'(x)}{I_\nu(x)}+\frac{K_\nu'(x)}{K_\nu(x)}&>-\frac{x}{x^2+\nu^2-1},\qquad&  &\nu\ge1 .
		\label{eq:sum-lowerbound}
	\end{alignat}
	Here, the constant $7/20$ in \eqref{eq:sum-upperbound} is the smallest possible such that the inequality holds for all $x > 0$ and all $\nu > -1$,
	and the constant $-1$ in \eqref{eq:sum-lowerbound} is the largest possible such that the inequality holds for all $x > 0$ and fixed $\nu \geq 1$.
	Moreover, for all $\nu \ge 1/2$ and $x>0$,
	\begin{equation}\label{eq:sum-extended}
		(x^2+\nu^2-1)\mleft(\frac{I_\nu'(x)}{I_\nu(x)}+\frac{K_\nu'(x)}{K_\nu(x)}\mright)>-x .
	\end{equation}
\end{theorem}
We prove these theorems in the same manner. The inequalities follow from the strong maximum principle (see Lemma \ref{lem:MaximumPrinciple}), and the sharpness follows from appropriate asymptotic formulas (see Section \ref{sec:sharp}).

Theorem \ref{thm:sum} leads us to a sharp monotonicity property of the product (Corollary \ref{cor:mono}): 
\begin{equation}
	\sqrt{x^2+\nu^2+7/20}\,I_\nu(x)K_\nu(x)
\end{equation}
is decreasing for $\nu>-1$ and
\begin{equation}
	\sqrt{x^2+\nu^2-1}\,I_\nu(x)K_\nu(x)
\end{equation}
is increasing for $\nu\ge1$, which improves the known facts that $I_\nu(x)K_\nu(x)$ is decreasing and $xI_\nu(x)K_\nu(x)$ is increasing \cite{Baricz:2009,Segura:2021}. This gives two-point inequalities for the product (Corollary \ref{cor:twopoint}) which, contrarily to the bounds for the product, are exact as the two points coalesce.

The paper is organized as follows. In Section \ref{sec:prelim} we collect the auxiliary results. Section \ref{sec:product} deals with the product and Section \ref{sec:sum} with the sum of logarithmic derivatives. In Section \ref{sec:sharp} we discuss the sharpness of the bounds by means of Debye-type expansions and the accuracy of the bounds. In Section \ref{sec:cons} we give some consequences: the sharp monotonicity properties of the product, two-point inequalities, and improved bounds for the ratios $I_{\nu-1}(x)/I_\nu(x)$ and $K_{\nu+1}(x)/K_\nu(x)$; and in Section \ref{sec:applications} we briefly mention two applications, related to the Casimir energy of a dielectric cylinder and to Schwarz domain decomposition methods.

\section{Preliminaries}\label{sec:prelim}
We will prove Theorems \ref{thm:product} and \ref{thm:sum} via the following strong maximum principle:
\begin{lemma}\label{lem:MaximumPrinciple}
	Assume that $u \in C^2((0, \infty))$ is not identically zero and satisfies the differential inequality
	\begin{equation} \label{eq:ODI}
		A(x) u'' + B(x) u' + C(x) u \geq 0
	\end{equation}
	on $(0, \infty)$, where $A$, $B$, $C$ are continuous functions satisfying $A > 0$ and $C \leq 0$ on $(0, \infty)$.
	If 
	\begin{equation} \label{eq:boundary}
		\limsup_{x \to 0} u(x) \leq 0 , \quad \limsup_{x \to \infty} u (x) \leq 0 , 
	\end{equation}
	then we have $u < 0$ on $(0, \infty)$.  
\end{lemma}
See \cite[Chapter 1]{ProtterWeinberger:1967}, for example.
In order to apply the maximum principle, we need the following elementary lemmas.
\begin{lemma}\label{lem:monotone}
	\begin{enumerate}[label=\textup{(M\arabic*)}]
		\item[]%
		\item\label{item:M1} $P_\nu(x)=I_\nu(x)K_\nu(x)$ is strictly positive, while its derivative is strictly negative on $(0,+\infty)$ for $\nu\ge -1$. In particular, the function $Q$ given in \eqref{eq:defH} is well defined, positive, and real analytic in $(0,+\infty)$.
		\item\label{item:M2} $xP_\nu(x)$ and its derivative are strictly positive on $(0,+\infty)$ for $\nu\ge 1/2$. Since $xP_\nu(x)\to 1/2$ as $x\to+\infty$, this implies $2xP_\nu(x)<1$, that is, $q(x)>x^2$, for $\nu\ge1/2$. Also, $(xP_\nu(x))'>0$ is equivalent to $Q(x) > x^2$.
	\end{enumerate}
\end{lemma}
\begin{lemma}\label{lem:limits}
	Let $\nu > -1$. Then we have 
	\begin{equation}\label{eq:limit at infty}
		h(x)=-\frac14+\frac{3}{2} (\nu^2 - 1/4) x^{-2} + \mathcal O(x^{-4}),\qquad H(x)\to-\frac14 
	\end{equation}
	as $x \to +\infty$, and 
	\begin{equation}\label{eq:limit at zero}
		h(0^+)=\begin{cases} 0, & \nu\ge 0,\\ -\nu^2, & -1<\nu<0,\end{cases}
		\qquad
		H(0^+)=\begin{cases} -1, & \nu \geq 1 ,\\ -\nu^2, & -1<\nu< 1.\end{cases}
	\end{equation}
	as $x \to 0^+$.
\end{lemma}
We use Lemmas \ref{lem:monotone} and \ref{lem:limits} to establish the differential inequality \eqref{eq:ODI} and to verify the boundary condition \eqref{eq:boundary}.
See \cite[Theorem 3 and its proof]{Segura:2021} for Lemma \ref{lem:monotone} (see also \cite{Baricz:2009} for earlier results).
Lemma \ref{lem:limits} can be shown by using some standard formulas for modified Bessel functions as follows.
\begin{proof}[Proof of Lemma \ref{lem:limits}]
	The large-argument behavior of $h$ and $H$ shown in \eqref{eq:limit at infty} follows from the standard large-argument asymptotic formulas for the modified Bessel functions and their derivatives; see \cite[10.40.1--4, 10.40.6]{NIST:DLMF}.
	The limit of $h$ as $x\to0^+$ in \eqref{eq:limit at zero} is also immediate from \cite[10.27.3, 10.30.1--3]{NIST:DLMF}
	\begin{equation}
		I_\nu(x) \sim \frac{x^\nu}{ 2^\nu \Gamma(\nu + 1) } , \qquad 
		K_\nu(x) \sim 
		\begin{cases}
			2^{ \lvert\nu\rvert - 1} \Gamma(\lvert \nu \rvert) x^{- \lvert \nu \rvert} , & \nu \ne 0 , \\
			- \log x , & \nu = 0 .
		\end{cases}
	\end{equation}
	
	For $H(0^+)$, the derivative formulas \cite[10.29.2]{NIST:DLMF} give
	\[
	\frac{d}{dx} \log (I_\nu(x) K_\nu(x))
	=\frac{I_{\nu+1}(x)}{I_\nu(x)}
	-\frac{K_{\nu-1}(x)}{K_\nu(x)} ,
	\]
	and the small-argument limits
	\cite[10.27.3, 10.30.1--3]{NIST:DLMF} imply
	\[
	\frac{I_{\nu+1}(x)}{x I_\nu(x)}
	\longrightarrow 
	\frac{1}{2 (\nu+1)},
	\qquad
	\frac{K_{\nu-1}(x)}{x K_\nu(x)}
	\longrightarrow
	\begin{cases}
		\dfrac{1}{2(\nu-1)}, & \nu > 1 , \\
		\infty, & -1 < \nu \leq 1 
	\end{cases}
	\]
	as $x \to 0^+$.
	Combining these, we obtain
	\begin{equation}
		\frac{1}{x} \frac{d}{dx} \log (I_\nu(x) K_\nu(x))
		\longrightarrow 	
		\begin{cases}
			- \dfrac{1}{\nu^2 - 1}, & \nu > 1 , \\
			- \infty, & -1 < \nu \leq 1 ,
		\end{cases}
	\end{equation}
	so that $H(x) = -x/((\log (I_\nu(x) K_\nu(x)))') - x^2 - \nu^2$ satisfies \eqref{eq:limit at zero}, as desired.
\end{proof}

\section{\texorpdfstring{The product $I_\nu(x)K_\nu(x)$}{The product}}\label{sec:product}
We recall the differential equation satisfied by $q = (2 I_\nu K_\nu)^{-2}$ \cite[Lemma 3]{SeguraSuzuki:2026}:
\begin{equation} \label{eq:odeq}
	4x^2 q q'' + 4x q q' - 5 x^2 (q')^2 = 16 q^2 (q - x^2 - \nu^2) .
\end{equation}
Substituting $h(x) = q(x) - x^2 - \nu^2$ into \eqref{eq:odeq}, we obtain the following immediately.
\begin{lemma}\label{lem:odeh}
	The function $h(x) = q(x) - x^2 - \nu^2$ solves the differential equation
	\begin{equation}\label{eq:odeh}
		A(x) h'' + B(x) h' + C(x) h = 20 x^4 - 16 x^2 q 
	\end{equation}
	on $(0, \infty)$, where
	\begin{equation} \label{eq:ABCh}
		A(x) = 4x^2 q, \quad 
		B(x) = 4xq - 5 x^2 q' - 10 x^3, \quad
		C(x) = - 16 q^2 ,
	\end{equation}
	which are continuous and satisfy $A > 0$ and $C \leq 0$.
\end{lemma}
Using this identity, Theorem \ref{thm:product} can be proved via the maximum principle (Lemma \ref{lem:MaximumPrinciple}) as follows.
\begin{proof}[Proof of Theorem \ref{thm:product}]
	Let $u(x) \coloneqq h(x) - 1/5$ and $v(x) \coloneqq - h(x) - 1/4$. 
	Then it suffices to show that $u < 0$ holds on $(0, \infty)$ when $\nu > -1$, and $v < 0$ holds on $(0, \infty)$ when $\nu \geq 1/2$.
	Thanks to Lemma \ref{lem:limits}, we know that $u$ and $v$ are not identically zero and satisfy
	$u(0^+)\le 0$ and $u(+\infty)\leq 0$ for $\nu > -1$, and $v(0^+)\le 0$ and $v(+\infty)\leq 0$ for $\nu \geq 1/2$.
	Moreover, using Lemma \ref{lem:odeh}, we obtain
	\begin{align*}
		A(x) u'' + B(x) u' + C(x) u
		\underset{\eqref{eq:odeh}}&= 20 x^4 - 16 x^2 q +\frac{16}{5} q^2  \\
		&= \frac{4}{5}(2q - 5x^2)^2 \geq 0 
	\end{align*}
	on $(0, \infty)$ for each $\nu > -1$.
	Similarly, using Lemmas \ref{lem:monotone} and \ref{lem:odeh}, we get
	\begin{align*}
		A(x) v'' + B(x) v' + C(x) v
		\underset{\eqref{eq:odeh}}&= - 20 x^4 + 16 x^2 q + 4q^2  \\
		&= 4 (q + 5x^2)(q - x^2) \underset{\text{\ref{item:M2}}}\geq 0 
	\end{align*}
	on $(0, \infty)$ for each $\nu \geq 1/2$.
	Hence, Theorem \ref{thm:product} follows from Lemma \ref{lem:MaximumPrinciple}.
\end{proof}
\begin{remark} \label{remark:product-upperbound}
	The constant $-1/4$ in \eqref{eq:product-upperbound} is best possible for every $\nu\ge1/2$, because $h(+\infty)=-1/4$. The range $\nu\ge1/2$ is also optimal: \eqref{eq:limit at infty} shows that $h(x)<-1/4$ for large $x$ when $-1/2 < \nu < 1/2$, and \eqref{eq:limit at zero} shows that $h(x)<-1/4$ for small $x$ when $-1 < \nu < -1/2$. Moreover, we have $I_{-1/2}(x) K_{-1/2}(x) = (1 + e^{-2x})/(2x)$, so that $h(x) < -1/4$ holds for every $x$ when $\nu = -1/2$. For instance, numerically $\inf_{x>0}h(x)\approx-0.3081$ for $\nu=0$.
\end{remark}

\section{The sum of logarithmic derivatives}\label{sec:sum}
We recall the differential equation satisfied by $P_\nu = I_\nu K_\nu$ \cite[Eq.~(10)]{SeguraSuzuki:2026}:
\begin{equation} \label{eq:odeP}
	(P_\nu')^2 - \frac{1}{x^2} 
	= 
	2 P_\nu P_\nu'' 
	+ \frac{2}{x} P_\nu P_\nu' 
	- 4\mleft( 1 + \frac{\nu^2}{x^2} \mright) P_\nu^2 .
\end{equation}
Using this equation, one can show the following.
\begin{lemma}\label{lem:odeH}
	The function $H(x)= Q(x) - x^2 - \nu^2$ solves the differential equation
	\begin{equation}\label{eq:odeH}
		A(x) H'' + B(x) H' + C(x) H = 15 x^4 - 18 x^2 Q + 4 Q^2
	\end{equation}
	on $(0, \infty)$, where $Q(x) = - x P_\nu(x) / P_\nu'(x)$ as in \eqref{eq:defH}, and
	\begin{equation} \label{eq:ABCH}
		A(x) = x^2 Q , \quad 
		B(x) = 5x Q - 2x^2 Q' - 7x^3 , \quad
		C(x) = - 4Q^2 ,
	\end{equation}
	which are continuous and satisfy $A > 0$ and $C \leq 0$.
\end{lemma}
\begin{proof}[Proof of Lemma \ref{lem:odeH}]
	Multiplying \eqref{eq:odeP} by $x^2$ and differentiating,
	we have
	\begin{equation} \label{eq:odeP-3rd-order}
		x^2 P_\nu''' + 3x P_\nu'' + (1 - 4(x^2+\nu^2)) P_\nu' - 4x P_\nu = 0.
	\end{equation}
	On the other hand, differentiating $P_\nu'=-xP_\nu/Q$ directly, we get
	\begin{align}
		\frac{P_\nu''}{P_\nu}
		&=\frac{x^2+xQ'-Q}{Q^2},\\
		\frac{P_\nu'''}{P_\nu}
		&=\frac{xQQ''+(2Q-2xQ'-3x^2)Q'+3xQ-x^3}{Q^3}.
	\end{align}
	Substituting
	these into \eqref{eq:odeP-3rd-order}, we obtain
	\begin{align}
		x^2 Q Q'' + (5xQ-2x^2Q'-3x^3)Q' - 4Q^2(Q-x^2-\nu^2) = x^4 - 6x^2 Q + 4Q^2 ,
	\end{align}
	which is equivalent to \eqref{eq:odeH} since $H = Q - x^2 - \nu^2$.
\end{proof}
Now we prove Theorem \ref{thm:sum}.
\begin{proof}[Proof of Theorem \ref{thm:sum}]
	Let $u(x) \coloneqq H(x) - 7/20$ and $v(x) \coloneqq - H(x) - 1$. 
	Then it suffices to show that $u < 0$ holds on $(0, \infty)$ when $\nu > -1$, and $v < 0$ holds on $(0, \infty)$ when $\nu \geq 1/2$.
	Thanks to Lemma \ref{lem:limits}, we know that $u$ and $v$ are not identically zero and satisfy
	$u(0^+)\le 0$ and $u(+\infty)\leq 0$ for $\nu > -1$, and $v(0^+)\le 0$ and $v(+\infty)\leq 0$ for $\nu \geq 1/2$.
	Moreover, using Lemma \ref{lem:odeH}, we obtain
	\begin{align*}
		A(x) u'' + B(x) u' + C(x) u \underset{\eqref{eq:odeH}}&= 15x^4 - 18x^2 Q + \frac{27}{5} Q^2 \\
		&= \frac{3}{5} (3Q - 5x^2)^2 \geq 0 
	\end{align*}
	on $(0, \infty)$ for each $\nu > -1$.
	Similarly, using Lemmas \ref{lem:monotone} and \ref{lem:odeH}, we get
	\begin{align*}
		A(x) v'' + B(x) v' + C(x) v 
		\underset{\eqref{eq:odeH}}&= - 15x^4 + 18x^2 Q \\
		&= 3 x^2 (6 Q - 5 x^2) \underset{\text{\ref{item:M2}}}\geq 0 
	\end{align*}
	on $(0, \infty)$ for each $\nu \geq 1/2$.
	Hence, Theorem \ref{thm:sum} follows from Lemma \ref{lem:MaximumPrinciple}.
\end{proof}
\begin{remark} \label{remark:sum-lowerbound}
	The constant $-1$ in \eqref{eq:sum-lowerbound} is best possible for every $\nu\ge1$, because $H(0^+)=-1$. The range $\nu\ge1$ is also optimal: if $- 1< \nu < 1$, then the right-hand side of \eqref{eq:sum-lowerbound} is positive for small $x$, while the left-hand side is negative. On the other hand, we do not know whether the range $\nu \geq 1/2$ for \eqref{eq:sum-extended} is optimal or not. Indeed, numerical calculation suggests that it is still valid for $-1 < \nu < 1/2$.
\end{remark}

\section{Sharpness}\label{sec:sharp}
As we saw in Remarks \ref{remark:product-upperbound} and \ref{remark:sum-lowerbound}, the constants $-1/4$ and $-1$ in \eqref{eq:product-upperbound} and \eqref{eq:sum-lowerbound} are optimal for each fixed $\nu$ because $h(+\infty)=-1/4$ and $H(0^+)=-1$, respectively. The constants $1/5$ and $7/20$ in \eqref{eq:product-lowerbound} and \eqref{eq:sum-upperbound} are not attained for fixed $\nu$, but are approached as $\nu\to+\infty$ along a ray $x=\nu z$. This can be seen from the Debye-type expansions \cite[10.41.3--6]{NIST:DLMF}, \cite[Theorem 3]{SeguraSuzuki:2026}
\begin{alignat}{1}
	I_\nu(\nu z)K_\nu(\nu z) 
	&= \frac{p}{2\nu} \mleft( 1 + \frac{a_2(p)}{\nu^2} +\mathcal O(\nu^{-4})\mright),
	\label{eq:debye} \\
	\frac{d}{dz} (I_\nu(\nu z)K_\nu(\nu z))
	&=- \frac{z p^3}{2\nu}
	\mleft(
	1+\frac{a_2(p) + p (d/dp) a_2(p)}{\nu^2}
	+\mathcal O(\nu^{-4})
	\mright)
	\label{eq:debye-derivative}
\end{alignat}
as $\nu \to \infty$ for fixed $z > 0$, where
\[
p=\frac{1}{\sqrt{1+z^2}} , \qquad a_2(p)=\frac{1}{8}p^2(p^2-1)(5p^2-1).
\]
Substituting \eqref{eq:debye} and \eqref{eq:debye-derivative} into the definitions of $h$ and $H$, we obtain
\begin{alignat}{4}
	h(\nu z)& ={} &h_0(p) &+ \mathcal O(\nu^{-2}), \qquad &h_0(p)&{}={}& & \frac{1}{4}(1-p^2)(5p^2-1), 
	\label{eq:h0}\\
	H(\nu z)& ={} &H_0(p) &+ \mathcal O(\nu^{-2}), \qquad &H_0(p)&{}={}&-& \frac{1}{4}(15p^4-12p^2+1), 
	\label{eq:H0}
\end{alignat}
as $\nu \to \infty$ for each fixed $z > 0$.

The function $h_0(p)$ attains its maximum $1/5$ in $[0,1]$ at $p^2=3/5$, that is, $z=\sqrt{2/3}$, and $H_0(p)$ attains its maximum $7/20$ at $p^2=2/5$, that is, $z=\sqrt{3/2}$. Therefore, we have
\[
\lim_{\nu\to+\infty}\sup_{x>0}h(x)=\frac15,\qquad \lim_{\nu\to+\infty}\sup_{x>0}H(x)=\frac7{20} .
\]
Table \ref{tab:max} shows the maximum values of $h$ and $H$ for several values of $\nu$, and the corresponding abscissas.

\begin{table}[ht]
	\centering
	\begin{tabular}{rcccc}
		\toprule
		$\nu$ & $\max_{x>0}h(x)$ & $x_{\max}/\nu$ & $\max_{x>0}H(x)$ & $x_{\max}/\nu$\\
		\midrule
		0 & 0.03995 & -- & 0.07325 & -- \\
		1/2 & 0.12737 & 1.1514 & 0.20716 & 2.1491\\
		1 & 0.16079 & 0.9531 & 0.26860 & 1.6200\\
		2 & 0.18428 & 0.8657 & 0.31662 & 1.3724\\
		5 & 0.19677 & 0.8267 & 0.34334 & 1.2554\\
		10 & 0.19915 & 0.8192 & 0.34829 & 1.2329\\
		50 & 0.19997 & 0.8166 & 0.34993 & 1.2251\\
		$\infty$ & 1/5 & $\sqrt{2/3}\approx0.8165$ & 7/20 & $\sqrt{3/2}\approx 1.2247$\\
		\bottomrule
	\end{tabular}
	\caption{Maximum values of $h(x)$ and $H(x)$ and the ratio $x/\nu$ where they are attained.}
	\label{tab:max}
\end{table}

\subsection{Accuracy of the bounds for the sum}\label{sec:accsum}

The bounds of Theorem \ref{thm:sum} are not only simple but also uniformly accurate: since $H$ is uniformly bounded, the relative error of both bounds is $\mathcal O(\nu^{-2})$ as $\nu\to+\infty$, uniformly in $x>0$. It is worth comparing them with the bounds that can be obtained from bounds for the ratios. We write
\begin{equation*}
	\phi_I(x) \coloneqq xI_{\nu-1}(x)/I_\nu(x), \qquad 
	\phi_K(x) \coloneqq xK_{\nu+1}(x)/K_\nu(x), 
\end{equation*}
so that
\begin{equation}
	\frac{d}{dx} \log (I_\nu(x) K_\nu(x)) = \frac{\phi_I(x)-\phi_K(x)}{x} .
\end{equation}
Recall that we have
\begin{equation}
	- 1 < x \frac{d}{dx} \log (I_\nu(x) K_\nu(x)) < 0
\end{equation}
for $\nu \geq 1/2$ and $x > 0$ thanks to Lemma \ref{lem:monotone}, while $\phi_I(\nu z)$ and $\phi_K(\nu z)$ are of order $\nu$ as $\nu \to \infty$ for fixed $z > 0$ due to Debye's asymptotic formulas \cite[10.41.3--6]{NIST:DLMF}. Thus, bounds for $\tfrac{d}{dx} \log (I_\nu(x) K_\nu(x))$ obtained as differences of bounds for the ratios suffer from cancellation. Table \ref{tab:comp} compares the relative errors of the bounds of Theorem \ref{thm:sum} with the best bounds obtained from the irrational bounds for the ratios of the form $\alpha+\sqrt{\beta^2+\gamma^2x^2}$ collected in \cite[Table 1]{SeguraSuzuki:2026}, which are of comparable simplicity.

\begin{table}[ht]
	\centering
	\begin{tabular}{ccccc}
		\toprule
		& \multicolumn{2}{c}{Theorem \ref{thm:sum}} & \multicolumn{2}{c}{From \cite[Table 1]{SeguraSuzuki:2026}}\\
		$x/\nu$ & lower & upper & lower & upper\\
		\midrule
		0.5 & $6.1\times10^{-3}$ & $4.7\times10^{-3}$ & $6.1\times10^{-2}$ & $9.6\times10^{-3}$\\
		$\sqrt{2/3}$ & $7.2\times10^{-3}$ & $9.2\times10^{-4}$ & $1.3\times10^{-1}$ & $4.7\times10^{-2}$\\
		$\sqrt{3/2}$ & $5.4\times10^{-3}$ & $7.0\times10^{-6}$ & $2.3\times10^{-1}$ & $1.5\times10^{-1}$\\
		3 & $1.0\times10^{-3}$ & $3.4\times10^{-4}$ & $1.1\times10^{-1}$ & $2.4\times10^{-1}$\\
		20 & $1.9\times10^{-5}$ & $1.5\times10^{-5}$ & $2.5\times10^{-3}$ & $4.8\times10^{-2}$\\
		\bottomrule
	\end{tabular}
	\caption{Relative errors of the lower and upper bounds for $x\mleft(I_\nu'(x)/I_\nu(x)+K_\nu'(x)/K_\nu(x)\mright)$ for $\nu=10$.}
	\label{tab:comp}
\end{table}

The bounds of Theorem \ref{thm:sum} are typically one or two orders of magnitude sharper. More accurate bounds for the ratios, such as the optimized bounds of \cite[Theorems 6 and 8]{SeguraSuzuki:2026}, lead to sharper estimates in part of the range, at the price of much more complicated expressions; even so, in a neighbourhood of the direction $x=\sqrt{3/2}\,\nu$ the upper bound of Theorem \ref{thm:sum} is the sharpest one (for $\nu=10$ and $x=\sqrt{3/2}\,\nu$ its relative error is $7.0\times10^{-6}$, while the best bound obtained from the ratios has relative error $5.4\times10^{-4}$), and its advantage increases with $\nu$.

\section{Consequences}\label{sec:cons}

\subsection{Monotonicity of the product}\label{sec:mono}

Theorem \ref{thm:sum} implies a monotonicity property of the product. Since, for any constant $c$ with $x^2+\nu^2+c>0$,
\begin{equation}\label{eq:dermono}
	\frac{d}{dx}\log\Bigl(\sqrt{x^2+\nu^2+c}\;I_\nu(x)K_\nu(x)\Bigr)=\frac{x\,\mleft(H(x)-c\mright)}{Q(x)\,(x^2+\nu^2+c)},
\end{equation}
the sign of this derivative is the sign of $H(x)-c$, and Theorem \ref{thm:sum} gives the following result.

\begin{corollary}\label{cor:mono}
	\begin{enumerate}[label=\textup{(\roman*)}]
		\item[]%
		\item For $\nu>-1$, the function $\sqrt{x^2+\nu^2+7/20}\;I_\nu(x)K_\nu(x)$ is strictly decreasing on $(0,+\infty)$.
		\item For $\nu\ge1$, the function $\sqrt{x^2+\nu^2-1}\;I_\nu(x)K_\nu(x)$ is strictly increasing on $(0,+\infty)$.
	\end{enumerate}
	Moreover, the constants are best possible in the following sense: if $c < 7/20$, then for each sufficiently large $\nu$, the function $\sqrt{x^2+\nu^2+c}\,I_\nu(x)K_\nu(x)$ is strictly increasing on some open interval $(a,b)$, and if $c>-1$, then it is strictly decreasing near the origin for each $\nu \geq 1$.
\end{corollary}
\begin{proof}[Proof of Corollary \ref{cor:mono}]
	(i) and (ii) follow from \eqref{eq:dermono} and Theorem \ref{thm:sum}. If $c < 7/20$, by Section \ref{sec:sharp} we have $\sup_{x>0}H(x)>c$ for $\nu$ large enough, and the derivative \eqref{eq:dermono} is positive at some points. If $c>-1$ and $\nu \geq 1$, $H(0^+)=-1<c$ and the derivative is negative for small $x$.
\end{proof}

Corollary \ref{cor:mono} sharpens two well known monotonicity properties of the product: (i) is stronger than the fact that $I_\nu(x)K_\nu(x)$ is decreasing for $\nu>-1$ \cite{Baricz:2009,Segura:2021}, and, because $x/\sqrt{x^2+\nu^2-1}$ is positive and non-decreasing for $\nu\ge1$, (ii) is stronger than the fact that $xI_\nu(x)K_\nu(x)$ is increasing \cite{Segura:2021} when $\nu\ge1$.

\subsection{Two-point inequalities}\label{sec:twopoint}

Integrating the bounds for the logarithmic derivatives we obtain inequalities relating the values of the functions at two different points. The bounds of Theorems \ref{thm:product} and \ref{thm:sum} can be integrated in terms of elementary functions. Let
\[
\Phi_\kappa(x)=\sqrt{x^2+m^2}-m\log\frac{m+\sqrt{x^2+m^2}}{x},\qquad m=\sqrt{\nu^2+\kappa},
\]
which satisfies $\Phi_\kappa'(x)=\sqrt{x^2+\nu^2+\kappa}/x$, and for $0<a<b$ write $\Delta_\kappa(a,b)=\Phi_\kappa(b)-\Phi_\kappa(a)$.
Notice that $\Phi_0(\nu z)= \nu \eta(z)$ for $\nu, z > 0$ with 
\begin{equation}
	\eta(z)=\sqrt{1+z^2}+\log\frac{z}{1+\sqrt{1+z^2}} ,
\end{equation}
which is the variable appearing in the Debye expansions \cite[10.41.3--4]{NIST:DLMF}.

\begin{corollary}\label{cor:twopoint}
	Let $0<a<b$. Then
	\begin{equation}\label{eq:twopoint}
		\sqrt{\frac{a^2+\nu^2-1}{b^2+\nu^2-1}}<\frac{I_\nu(b)K_\nu(b)}{I_\nu(a)K_\nu(a)}<\sqrt{\frac{a^2+\nu^2+7/20}{b^2+\nu^2+7/20}},
	\end{equation}
	where the upper bound holds for $\nu>-1$ and the lower bound for $\nu\ge1$, and
	\begin{equation}\label{eq:rho}
		\exp\mleft(2\Delta_{-1/4}(a,b)\mright)<\frac{I_\nu(b)K_\nu(a)}{I_\nu(a)K_\nu(b)}<\exp\mleft(2\Delta_{1/5}(a,b)\mright),
	\end{equation}
	where the upper bound holds for $\nu>-1$ and the lower bound for $\nu\ge1/2$.
\end{corollary}
\begin{proof}[Proof of Corollary \ref{cor:twopoint}]
	\eqref{eq:twopoint} follows from Corollary \ref{cor:mono}. For \eqref{eq:rho}, we recall
	\begin{equation*}
		\frac{d}{dx}\log\frac{I_\nu(x)}{K_\nu(x)}=\ell_I-\ell_K=1/(xI_\nu(x) K_\nu(x)) ,
	\end{equation*}
	so that Theorem \ref{thm:product} can be written as  
	\begin{equation*}
		\frac{2}{x} \sqrt{x^2+\nu^2-1/4} <\frac{d}{dx}\log\frac{I_\nu(x)}{K_\nu(x)}< \frac{2}{x} \sqrt{x^2+\nu^2+1/5} ,
	\end{equation*}
	where the upper and lower bounds hold for $\nu > -1$ and $\nu \geq 1/2$, respectively.
	Integrating this inequality between $a$ and $b$, we get \eqref{eq:rho}.
\end{proof}

The inequality \eqref{eq:twopoint} improves the bounds $a/b<P_\nu(b)/P_\nu(a)<1$ which follow from the monotonicity of $xP_\nu(x)$ and $P_\nu(x)$. The bounds of Corollary \ref{cor:twopoint} tend to $1$ as $b\to a$, and are therefore accurate when $a$ and $b$ are close. For \eqref{eq:twopoint} this is a property that cannot be obtained from bounds for the function $I_\nu(x)K_\nu(x)$: for instance, Theorem \ref{thm:product} only gives
\[
\frac{I_\nu(b)K_\nu(b)}{I_\nu(a)K_\nu(a)}<\sqrt{\frac{a^2+\nu^2+1/5}{b^2+\nu^2-1/4}},
\]
which is larger than $1$ when $b$ is close to $a$, and the same happens with any bounds for the product obtained from bounds for the ratios $I_{\nu-1}(x)/I_\nu(x)$ and $K_{\nu+1}(x)/K_\nu(x)$; bounds for the logarithmic derivative, such as those of Theorem \ref{thm:sum}, are needed. In addition, \eqref{eq:twopoint} is very accurate for large $\nu$: since the relative error of the bounds of Theorem \ref{thm:sum} is $\mathcal O(\nu^{-2})$ uniformly in $x$ and $P_\nu'/P_\nu$ is always negative, the logarithms of the bounds in \eqref{eq:twopoint} approximate $\log\mleft(P_\nu(b)/P_\nu(a)\mright)$ with relative error $\mathcal O(\nu^{-2})$, uniformly in $0<a<b$ (for instance, for $\nu=10$ it is smaller than $1.4\times10^{-2}$).

\subsection{Improved bounds for the ratios}\label{sec:ratios}
Using the relations \cite[10.29.2]{NIST:DLMF}
\begin{align*}
	xI_\nu'(x)&=xI_{\nu-1}(x)-\nu I_\nu(x) , \\
	xK_\nu'(x)&=\nu K_\nu(x)-xK_{\nu+1}(x) ,
\end{align*}
we can rewrite Theorem \ref{thm:sum}  as
\begin{equation}\label{eq:diffratios}
	\frac{x^2}{x^2+\nu^2+7/20}<x\frac{K_{\nu+1}(x)}{K_\nu(x)}-x\frac{I_{\nu-1}(x)}{I_\nu(x)}<\frac{x^2}{x^2+\nu^2-1},
\end{equation}
where the lower bound holds for $\nu>-1$ and the upper bound for $\nu\ge1$. This can be used to improve known bounds for the ratios $I_{\nu-1}(x)/I_\nu(x)$ and $K_{\nu+1}(x)/K_\nu(x)$. In \cite[Theorem 5]{Segura:2023} the following bounds were obtained for $\nu\ge1/2$ and $x>0$:
\begin{alignat}{3}
	\nu+\sqrt{\nu^2+x(x-1)}<{}& x\frac{I_{\nu-1}(x)}{I_\nu(x)}&{}\le{}& \nu+\sqrt{\nu^2+x(x+1)},\label{eq:S23I}\\
	\nu+\sqrt{\nu^2+x(x-1)}<{}& x\frac{K_{\nu+1}(x)}{K_\nu(x)} &{}\le{}& \nu+\sqrt{\nu^2+x(x+1)}.\label{eq:S23K}
\end{alignat}
With the notation of \cite{Segura:2023} (see also \cite{SeguraSuzuki:2026}), the lower bound in \eqref{eq:S23I} and the upper bound in \eqref{eq:S23K} have accuracy $(1,3)$, while the other two bounds have accuracy $(1,1)$, that is, they only reproduce the leading term of the expansion as $x\to+\infty$. Combining the $(1,3)$ bounds with the lower bound in \eqref{eq:diffratios} we obtain improved versions of the $(1,1)$ bounds.

\begin{corollary}\label{cor:improved}
	For $\nu\ge1/2$ and $x>0$,
	\begin{align*}
		x\frac{I_{\nu-1}(x)}{I_\nu(x)}&<\nu+\sqrt{\nu^2+x(x+1)}-\frac{x^2}{x^2+\nu^2+7/20},\\
		x\frac{K_{\nu+1}(x)}{K_\nu(x)}&>\nu+\sqrt{\nu^2+x(x-1)}+\frac{x^2}{x^2+\nu^2+7/20}.
	\end{align*}
	These bounds are sharper than the upper bound in \eqref{eq:S23I} and the lower bound in \eqref{eq:S23K}, and they have accuracy $(1,3)$.
\end{corollary}
\begin{proof}[Proof of Corollary \ref{cor:improved}]
	The first inequality follows from 
	\begin{equation*}
		x \frac{I_{\nu-1}}{I_\nu} = x \frac{K_{\nu+1}}{K_\nu} - x \mleft(\frac{K_{\nu+1}}{K_\nu} - \frac{I_{\nu-1}}{I_\nu} \mright) ,
	\end{equation*}
	the upper bound in \eqref{eq:S23K} and the lower bound in \eqref{eq:diffratios}. The second one follows similarly from 
	\begin{equation*}
		x \frac{K_{\nu+1}}{K_\nu} = x \frac{I_{\nu-1}}{I_\nu} + x \mleft(\frac{K_{\nu+1}}{K_\nu} - \frac{I_{\nu-1}}{I_\nu} \mright) ,
	\end{equation*}
	and the lower bound in \eqref{eq:S23I} and the lower bound in \eqref{eq:diffratios}. Because 
	\begin{equation}
		\frac{x^2}{x^2+\nu^2+7/20}= 1- (\nu^2+ 7/20) x^{-2} + \mathcal O(x^{-4})
	\end{equation}
	as $x\to+\infty$, the new bounds reproduce the first three terms of the expansions of the ratios as $x\to+\infty$, as the $(1,3)$ bounds used in their derivation.
\end{proof}

The improvement is also significant for large $\nu$ with $x/\nu$ fixed. For instance, for $x=\nu=10$, the relative error of the lower bound for $xK_{\nu+1}(x)/K_\nu(x)$ in \eqref{eq:S23K} is $2.5\times10^{-2}$, while that of Corollary \ref{cor:improved} is $4.7\times 10^{-3}$; for $x=\nu=100$ the relative errors are $2.5\times10^{-3}$ and $4.3\times 10^{-4}$, respectively.

\section{Applications}\label{sec:applications}

We briefly mention two applications in which the bounds of Theorems \ref{thm:product} and \ref{thm:sum} give elementary and uniformly accurate estimates.

\subsection{Casimir energy of a dielectric cylinder}

The function
\[
\mu_n(x)=x\,\frac{d}{dx}\mleft(I_n(x)K_n(x)\mright)=x\,I_n(x)K_n(x)\,\frac{d}{dx}\log(I_n(x)K_n(x))
\]
appears in the computation of the Casimir energy of a dielectric cylinder of radius $a$ when the velocity of light is the same inside and outside the cylinder: the mode sum giving the energy involves the terms $\log\bigl[1-\xi^2\mu_n(ay)^2\bigr]$, $n\in\mathbb Z$, where $\xi$ depends on the permittivities of both media \cite{Lambiase:1999,Klich:2000}. Since $\mu_{-n}=\mu_n$, it is enough to consider $n\ge0$, and Theorems \ref{thm:product} and \ref{thm:sum} give
\[
\frac{x^2}{2\mleft(x^2+n^2+7/20\mright)\sqrt{x^2+n^2+1/5}}<-\mu_n(x)<\frac{x^2}{2\mleft(x^2+n^2-1\mright)\sqrt{x^2+n^2-1/4}},
\]
where the lower bound holds for $n\ge0$ and the upper bound for $n\ge1$. Both bounds have relative error $\mathcal O(n^{-2})$ as $n\to\infty$, uniformly in $x>0$, and they reproduce, with rigorous two-sided control, the leading term $\mu_n(nz)^2\sim p^6z^4/(4n^2)$, $p=(1+z^2)^{-1/2}$, of the uniform asymptotic expansion used for large $n$ in \cite[Eq.~(3.13)]{Lambiase:1999}.

\subsection{Schwarz methods for problems with cylindrical interfaces}

In the analysis of Schwarz domain decomposition methods for the problem $(\eta-\Delta)u=f$ with cylindrical interfaces $r=a$ and $r=b$, $a<b$ \cite{Gigante:2013}, the Fourier modes of the solutions are multiples of $I_m(\alpha r)$ and $K_m(\alpha r)$, with $\alpha=\sqrt{k^2+\eta}$, $m\in\mathbb Z$, $k\in\mathbb R$. The reduction factor of the classical Schwarz method is the inverse of the quantity bounded in \eqref{eq:rho} \cite[Proposition 3.1]{Gigante:2013}, and therefore
\[
\exp\mleft(-2\Delta_{1/5}(\alpha a,\alpha b)\mright)<\rho_{\rm cla}(m,k)=\frac{I_m(\alpha a)\,K_m(\alpha b)}{I_m(\alpha b)\,K_m(\alpha a)}<\exp\mleft(-2\Delta_{-1/4}(\alpha a,\alpha b)\mright),
\]
the upper bound for $m\ge1$ and the lower bound for $m\ge0$ (with $\nu=m$ in $\Delta_\kappa$). The analysis of the optimized method in \cite[Section 4.4]{Gigante:2013} relies on the quantities $\mathcal A=\alpha I_m'(\alpha a)/I_m(\alpha a)$ and $\mathcal B=-\alpha K_m'(\alpha a)/K_m(\alpha a)$, their monotonicity with respect to the frequency $k$, and their values at the extreme frequencies \cite[Lemma 4.4 and Theorem 4.5]{Gigante:2013}. By the Wronskian relation,
\[
\mathcal A+\mathcal B=\frac{1}{a\,I_m(\alpha a)K_m(\alpha a)},
\]
and Corollary \ref{cor:twopoint} gives, for $|k_1|<|k_2|$ and $x_j=a\sqrt{k_j^2+\eta}$,
\[
\sqrt{\frac{x_2^2+m^2+7/20}{x_1^2+m^2+7/20}}<\frac{(\mathcal A+\mathcal B)(m,k_2)}{(\mathcal A+\mathcal B)(m,k_1)}<\sqrt{\frac{x_2^2+m^2-1}{x_1^2+m^2-1}},
\]
the lower bound for $m\ge0$ and the upper bound for $m\ge1$; as discussed after Corollary \ref{cor:twopoint}, estimates of this type, which remain accurate when $k_1$ and $k_2$ are close, cannot be obtained from bounds for the product $I_m(x)K_m(x)$ or for the ratios of modified Bessel functions.

\section*{Use of Artificial Intelligence}
During the preparation of this work, the authors used Claude Sonnet 5.5 (Anthropic) in order to assist with drafting the manuscript and performing numerical checks. All theorem statements and proofs were developed entirely by the authors. The authors reviewed and edited the AI-assisted content as needed and take full responsibility for the content of the manuscript.

\bibliographystyle{amsplain}
\bibliography{bessel_product_sum.bib}

\end{document}